\documentclass{amsart}

\usepackage{amsmath,amsthm,txfonts}

\def\cprime{$'$}
\newtheorem{thm}{Theorem}
\newtheorem*{rem}{Remark}

\newcommand{\DF}{\mathcal{E}}
\DeclareMathOperator{\Lip}{Lip}

\title{On a theorem of Grigor\cprime yan, Hu and Lau}
\author{Luke~G Rogers}

\begin{document}
\begin{abstract}
We refine a  result of Grigor\cprime yan, Hu and Lau to give a moment condition on a heat kernel which characterizes the critical exponent at which a family of Besov spaces associated to the Dirichlet energy becomes trivial.
\end{abstract}
\maketitle

Following~\cite{GHL}, consider a metric measure space $(M,d,\mu)$, with $M$ nonempty and $\mu$ Borel, that admits a heat kernel $\{p_{t}\}_{t>0}$.  The latter is a collection of symmetric, non-negative, measurable functions on $M\times M$ all of which have unit integral, satisfy the semigroup property $p_{t+s}(x,y)=\int p_{s}(x,z)p_{t}(z,y)d\mu(z)$ for all $s,t>0$, and approximate the identity in the sense that if $f\in L^2$ then $\int p_{t}(x,y)f(y)d\mu(y)\to f$ in $L^2$ as $t\downarrow0$.  This hypothesis has many consequences, among which we will need that setting
\begin{gather*}
	T_{t}u(x)=\int_{M} p_{t}(x,y) u(y)\, d\mu(y), \text{and}\\
	\DF_{t}(u) =t^{-1} \langle u- T_{t}u, u \rangle,
	\end{gather*}
where $\langle,\rangle$ is the $L^2$ inner product, we find $\DF_{t}(u)$ is decreasing in $t>0$ so $\DF(u)=\lim_{t\downarrow0} \DF_{t}(u)$ exists, though it may be infinite.  Moreover setting $\mathcal{D}(\DF)=\{u\in L^{2}: \DF(u)<\infty\}$ we have that $\DF$ is a DIrichlet form with domain $\mathcal{D}(\DF)$.

In~\cite{GHL} the authors further assume the heat kernel has a two-sided estimate of the form
\begin{equation}\label{eqn:HK}
\frac{1}{t^{\alpha/\beta}} \Phi_{1}\biggl(\frac{d(x,y)}{t^{1/\beta}}\biggr)
	\leq p_{t}(x,y)
	\leq \frac{1}{t^{\alpha/\beta}} \Phi_{2}\biggl(\frac{d(x,y)}{t^{1/\beta}}\biggr)
\end{equation}
for $\mu$-a.e. $x,y\in M$ and all $t>0$, where $\alpha$ and $\beta$ are positive constants and $\Phi_{1}$ and $\Phi_{2}$ are non-negative monotone decreasing functions on $[0,\infty)$.   They then show that $\alpha$ and $\beta$ are determined by $(M,d,\mu)$ provided that $\Phi_{1}(1)>0$ and $\Phi_{2}$ satisfies a moment condition of the form
\begin{equation}\label{eqn:Hgamma}
	\int_{0}^{\infty} s^{\gamma} \Phi_{2}(s)\frac{ds}{s} <\infty \tag{$H_{\gamma}$}
	\end{equation}
for some suitable value of $\gamma$.  Some of these results are stated in terms of a Besov space which they denote $W^{\sigma,2}$ but which is sometimes called $\Lip(\sigma,2,\infty)$. To define this space let
\begin{equation*}
	W_{\sigma}(u)
	= \sup_{0<r<1} r^{-2\sigma} \int_{M} \fint_{B(x,r)} |u(y)-u(x)|^{2}\,d\mu(y)\, d\mu(x)
\end{equation*}
where $\fint_{B}=\mu(B)^{-1}\int_{B}$ is the average and $B(x,r)$ is the ball of radius $r$ with center $x$.  Then let $W^{\sigma,2}=\{u\in L^{2}: W_{\sigma}(u)<\infty\}$.  This is a Banach space with norm $\|u\|_{2}+W_{\sigma}^{1/2}$.  Also let $\beta^{\ast}=2\sup\{\sigma:\dim W^{\sigma,2}=\infty\}$.

Among the main results in~\cite{GHL} are the following:
\begin{thm}[\protect{\cite{GHL}} Theorems 3.2, 4.2, 4.6]\label{thm:GHL}\ \\
Suppose $(M,d,\mu)$ has a heat kernel satisfying~\eqref{eqn:HK} and that $\Phi_{1}(1)>0$.
\begin{enumerate}
\item If~\eqref{eqn:Hgamma} holds for $\gamma=\alpha$ then $\mu$ is Ahlfors regular with exponent $\alpha$.
\item If~\eqref{eqn:Hgamma} holds for $\gamma=\alpha+\beta$ then $\mathcal{D}(\DF)=W^{\beta/2,2}$ and $\DF(u)\simeq W_{\beta/2}(u)$.
\item If~\eqref{eqn:Hgamma} holds for $\gamma>\alpha+\beta$ then for $\sigma>\beta/2$ the space $W^{\sigma,2}=\{0\}$ and $\beta=\beta^{\ast}$.
\end{enumerate}
\end{thm}
The purpose of this note is to show that the third of the above implications may be improved as follows:
\begin{thm}
Suppose $(M,d,\mu)$ has a heat kernel satisfying~\eqref{eqn:HK},  that $\Phi_{1}(1)>0$  and~\eqref{eqn:Hgamma} holds for $\gamma=\alpha+\beta$.  Then  $W^{\sigma,2}=\{0\}$ and $\beta=\beta^{\ast}$.
\end{thm}
\begin{proof}
We follow the proof of Theorem~4.6 of~\cite{GHL}.  They decompose $\DF_{t}(u)=A(t)+B(t)$ where for $\epsilon=\sigma-\beta$
\begin{align*}
	B(t) 
	&= \frac{1}{2t}\int_{M}\int_{B(x,1)} (u(x)-u(y))^2 p_{t}(x,y) \, d\mu(y)\, d\mu(x)\\
	&=  \frac{1}{2t}\sum_{k=1}^{\infty} \int_{M}\int_{ B(x,2^{-(k-1)}\setminus B(x,2^{-k}) } (u(x)-u(y))^2 p_{t}(x,y) \, d\mu(y)\, d\mu(x)\\
	&\leq \frac{t^{\epsilon/\beta}}{2} \sum_{k=1}^{\infty} \biggl( \frac{2^{-k}}{t^{1/\beta}}\biggr)^{\alpha+\beta+\epsilon}  \Phi_{2}\biggl(\frac{2^{-k}}{t^{1/\beta}}\biggr) 2^{k(\alpha+\sigma)} \int_{M}\int_{B(x,2^{-(k-1)})}  (u(x)-u(y))^2 \, d\mu(y)\, d\mu(x)\\
	&\leq C W_{\sigma}(u) t^{\epsilon/\beta} \sum_{k=1}^{\infty} \biggl( \frac{2^{-k}}{t^{1/\beta}}\biggr)^{\alpha+\beta+\epsilon}  \Phi_{2}\biggl(\frac{2^{-k}}{t^{1/\beta}}\biggr)\\
	&\leq C W_{\sigma}(u)  t^{\epsilon/\beta} \int_{0}^{t^{-1/\beta}} s^{\alpha+\beta+\epsilon} \Phi_{2}(s)\, \frac{ds}{s}
	\end{align*}
in which the first inequality is from the upper bound in~\eqref{eqn:HK} and the second is from  the definition of $W_{\sigma}(u)$ and the fact that part (1) of Theorem~\ref{thm:GHL} implies $\mu(B(x,2^{-(k-1)}))\simeq 2^{k\alpha}$.

The above is essentially shown in the proof of  Theorem~4.6 in~\cite{GHL}; they then assume~\eqref{eqn:Hgamma} for $\gamma=\alpha+\beta+\epsilon$ to establish that the integral is bounded independent of $t$ and conclude $\lim_{t\downarrow0}B(t)=0$.  However this also follows from~\eqref{eqn:Hgamma} for $\gamma=\alpha+\beta$.  This is actually a standard exercise: given $\delta>0$ use~\eqref{eqn:Hgamma} for $\gamma=\alpha+\beta$ to obtain $T$ so small that 
\begin{equation*}
\int_{T^{-1/\beta}}^{\infty} s^{\alpha+\beta} \Phi_{2}(s)\, \frac{ds}{s}<\delta
\end{equation*}
from which
\begin{align*}
	B(t)
	&\leq CW_{\sigma}(u) t^{\epsilon/\beta}\int_{0}^{T^{-1/\beta}} s^{\alpha+\beta+\epsilon} \Phi_{2}(s)\, \frac{ds}{s} +  CW_{\sigma}(u) t^{\epsilon/\beta}\int_{T^{-1/\beta}}^{t^{-1/\beta}} s^{\alpha+\beta+\epsilon} \Phi_{2}(s)\, \frac{ds}{s}\\
	&\leq  CW_{\sigma}(u) \biggl( \frac{t}{T} \biggr)^{\epsilon/\beta}\int_{0}^{T^{-1/\beta}} s^{\alpha+\beta} \Phi_{2}(s)\, \frac{ds}{s}
		+ CW_{\sigma}(u)\int_{T^{-1/\beta}}^{t^{-1/\beta}} s^{\alpha+\beta} \Phi_{2}(s)\, \frac{ds}{s}\\
	&\leq  CW_{\sigma}(u) \biggl( \frac{t}{T} \biggr)^{\epsilon/\beta}\int_{0}^{\infty} s^{\alpha+\beta} \Phi_{2}(s)\, \frac{ds}{s} + CW_{\sigma}(u)\delta
	\end{align*}
and $\lim_{t\downarrow0}B(t)$ follows.  Since it is established in  equation (4.17) of~\cite{GHL} that $\lim_{t\downarrow0} A(t)=0$ we conclude
\begin{equation*}
	\DF(u)
	= \lim_{t\downarrow0} \DF_{t}(u)
	= \lim_{t\downarrow0}  A(t)+B(t)\\
	=0.
	\end{equation*}
\end{proof}

\begin{rem}
A similar argument is used for a slightly different purpose in~\cite{P2}, and a slightly less general result with the same proof is in~\cite{P1}.  Nonetheless this specific result does not seem to be known, and the weaker result in part (3) of Theorem~\ref{thm:GHL} is frequently cited.
\end{rem}


\begin{thebibliography}{10}

\bibitem{GHL}
Alexander Grigor\cprime yan, Jiaxin Hu and Ka-Sing Lau
\emph{Heat kernels on metric measure spaces and an application to semilinear elliptic equations.}
Trans. Amer. Math. Soc. \textbf{355} (2003), no. 5, 2065–2095.

\bibitem{P1}
Katarzyna Pietruska-Pa\l uba
\emph{On function spaces related to fractional diffusions on d-sets}
Stoch. Stoch. Rep. \textbf{70}, 153–164 (2000)

\bibitem{P2}
Katarzyna Pietruska-Pa\l uba
\emph{Heat kernels on metric spaces and a characterisation of constant functions}
Man. Math. \textbf{115}, 389–399 (2004)





\end{thebibliography}
\end{document}